\documentclass[12pt,a4paper]{amsart}
\usepackage{amsthm}
\usepackage{amsmath}
\usepackage{amsfonts}
\usepackage{amssymb}
\usepackage[left=2.8cm, right=2.8cm, top=3.0cm, bottom=3.0cm]{geometry}
\usepackage{color}
\usepackage[arrow, matrix, curve]{xy}
\usepackage[colorlinks=true,linkcolor=green]{hyperref}
\makeatletter
\renewcommand*{\eqref}[1]{
	\hyperref[{#1}]{\textup{\tagform@{\ref*{#1}}}}}
\makeatother
\newcommand{\ind}{{\rm ind}}

\newcommand{\beq}{\begin{equation}}
\newcommand{\beqn}{\begin{equation*}}
\newcommand{\birk}{L}

\newcommand{\dil}{D}
\newcommand{\eeq}{\end{equation}}
\newcommand{\eeqn}{\end{equation*}}
\newcommand{\la}{\Lambda}
\newcommand{\ola}{\overline{\Lambda}}

\newcommand{\len}{L}
\newcommand{\prm}{p_m}
\newcommand{\R}{\mathbb{R}}

\newcommand{\leins}{L_1}
\newcommand{\lm}{L_3}

\newcommand{\n}{\mathbb{N}}
\newcommand{\q}{\mathbb{Q}}
\newcommand{\rk}{{\rm rk}}

\newcommand{\sn}{S^n}
\newcommand{\sm}{S^{2m-1}}
\newcommand{\sone}{S^1}
\newcommand{\z}{\mathbb{Z}}
\theoremstyle{plain}
\newtheorem{theorem}{Theorem}[section]

\newtheorem{corollary}[theorem]{Corollary}
\newtheorem{lemma}[theorem]{Lemma}
\newtheorem{assumption}[theorem]{Assumption}
\theoremstyle{definition}
\theoremstyle{remark}
\newtheorem{remark}[theorem]{Remark}
\newtheorem{number-env}[theorem]{}
\title{ Two short closed geodesics 
		on a sphere of odd dimension}
\author{Hans-Bert Rademacher}
\address{Mathematisches Institut, 
Universit{\"a}t Leipzig, D--04081 Leipzig, Germany}
\email{rademacher@math.uni-leipzig.de}
\urladdr{www.math.uni-leipzig.de/\symbol{126}rademacher}
\date{revised version, 2023-01-18}
\begin{document}
%%%%%%%%%%%%%%%%%%%%%%%%%%%%%%%%%%%%%%%%%%%%%%%%%%%%
\begin{abstract}
	We show that for an open and dense
	set of \emph{non-reversible} Finsler metrics
	on a sphere $S^n$ of odd dimension
	$n=2m-1\ge 3$ there is a second
	closed geodesic
	with Morse index
	$\le 4(m+2)(m-1)+2.$	
\end{abstract}
%%%%%%%%%%%%%%%%%%%%%%%%%%%%%%%%%%%%%%%%%%%%%%%%%%%%
\maketitle
%%%%%%%%%%%%%%%%%%%%
\baselineskip 18pt
%%%%%%%%%%%%%%%%%%%%%%%%%%%%%%%%%%%%%%%%%%%%%%%%%%%%%
\section{Introduction}
%%%%%%%%%%%%%%%%%%%%%%%%%%%%%%%%%%%%%%%%%%%%%%%%%%%%
In this paper we consider the
sphere $S^n$ 
of dimension $n \ge 2$
carrying a
non-reversible Finsler metric $f.$
Hence the length of a curve in general depends
on the orientation. The 
\emph{reversibility}
$\lambda=\max\{f(-X)\,;\, f(X)=1\}$
was introduced in \cite{Ra04}.
Then $\lambda \ge 1$
and $\lambda=1$ if and only if the Finsler
metric is \emph{reversible,}
i.e. $f(-X)=f(X)$ for all tangent vectors $X.$
For a tangent vector $X \in TS^n$
 we denote by $f_0(X)=\sqrt{g_0(X,X)}$ the length of a vector 
with respect to the standard Riemannian metric 
$g_0$ of constant sectional curvature $1$
on $S^n.$
%%%%%%%%%%%%%%%%%%%%%%%%%%%%%%%%%%%%%%%%%%%%%%%%%%%%%%%
Let $\dil=\dil(f)$ be the smallest positive number such that
\beq
\label{eq:dil}
\dil^{-1} f_0(X)\le f(X)\le \dil f_0 (X)
\eeq
holds for all tangent vectors $X.$
We call this invariant the \emph{distortion} of the Finsler metric $f.$
Obviously $\dil^2 \ge \lambda.$
%%%%%%%%%%%%%%%%%%%%%%%%%%%%%%%%%%%%%%%%%%%%%%%%%%%%%%%
Let $\birk=\birk(f)$ be the critical value of 
a generator of the non-trivial  homology
class $H_{n-1}\left(\Lambda S^n/S^1;\mathbb{Q}\right)\cong \mathbb{Q}$
in dimension $(n-1)$ in the free loop space $\Lambda S^n.$
Lyusternik and Fet~\cite{LF} used an idea by Birkhoff 
to show the existence of a closed geodesic 
$c_1$ whose length $l(c_1)$ equals $\birk$ 
and whose Morse index satisfies $\ind (c_1)\le n-1.$
Inequality~\eqref{eq:dil} implies that 
$2\pi /\dil\le L= l (c_1)\le 2\pi \dil.$
It follows from a result by 
Fet~\cite{Fe} that there exists a
second closed geodesic for a 
\emph{reversible} Finsler metric
which is bumpy, i.e. all its closed geodesics are
non-degenerate.  

In this paper we consider the existence of a second
closed geodesic for a \emph{non-reversible} Finsler metric.
On a $2$-sphere with a bumpy metric there always exists
a second closed geodesic $c_2$ 
geometrically distinct from $c_1$ as shown in \cite[(4.1)]{Ra89}.
Bangert and Long were able to show in \cite{BL} that this
statement holds for \emph{any} 
non-reversible Finsler metric.
There is a family $f_{\mu}, \mu \in [0,1), \mu \not\in \q$ of  Katok metrics on $S^2$ 
which are bumpy, have constant flag curvature
$1$ and carry exactly two
geometrically distinct closed geodesics
$c_1,c_2$ with 
$\ind(c_1)=1,\lim_{\mu \to 1}\ind(c_2)=\infty,
L(c_1)< 2\pi,$ and $\lim_{\mu \to 1}L(c_2)=\infty.$
Hence there exists 
in general only one \emph{short} closed geodesic
on $S^2.$

In higher dimensions there are many 
results on the existence of a second closed
geodesic, cf. for example 
\cite{DL},\cite{Ra10}, \cite{Ra17},
\cite[Cor.1.2]{DLW}, and
~\cite[Cor.1.14]{AGKM}.
Compare also the recent survey~\cite{LW2022}.
For existence results for closed geodesics
in Riemannian and Finsler geometry we also refer to the surveys
\cite{T2010} and \cite{Oancea}.
Under curvature assumptions one can give bounds for
the index of the second closed geodesic,
cf. for example \cite{Ra2007}.
But we are not aware of
estimates for the
index of the second closed geodesic
holding on an open and dense subset of metrics
on an $n$-dimensional sphere with $n \ge 3.$

We state our main result which shows 
in particular that for
an odd-dimensional sphere of dimension
$n=2m-1\ge 3$ endowed with a
bumpy metric there are two
geometrically distinct short closed geodesics,
with index $\le 4(m+2)(m-1)+2.$  More precisely we show:
%%%%%%%%%%%%%%%%%%%%%%%%%%%%%%%%%%%%%%%%%%%%%%%%%%%%%%
%%%%%%%%%%%%%%%%%%%%%%%%%%%%%%%%%%%%%%%%%%%%%%%%%%%%%%%
\begin{theorem}
\label{thm:length}
	Let $f$ be a non-reversible Finsler metric on the
	odd-dimensional sphere $S^{n}$ of dimension
	$n=2m-1\ge 3$
	with 
	distortion $\dil =\dil (f)\,.$ 
	Let $\prm$ be the smallest prime number which 
	is neither a divisor of $(m-1)$ nor of $m,$
	cf.	Lemma~\ref{lem:pm}, in particular 
	 $3 \le \prm\le m+2$
	for all $m \ge 2.$
	Assume that  all closed geodesics with length
	$\le \lm:=2\pi \prm D^3$ are
	non-degenerate.
	Then there are two geometrically
	distinct closed geodesics 
	with index $\le 4 \prm (m-1)+2$
	and 
	of  length 
	$\le \lm\,.$
\end{theorem}
%%%%%%%%%%%%%%%%%%%%%%%%%%%%%%%%%%%%%%%%%%%%%%%%%%%%%%
%%%%%%%%%%%%%%%%%%%%%%%%%%%%%%%%%%%%%%%%%%%%%%%%%%%%%%
For $n=3,m=2, p_2=3$ we obtain for the second closed
geodesic $c_2$ on $S^3:$
$\ind c_2\le 14.$
For $n=6k+3=2m-1$ resp.
$m \equiv 2\pmod{3}$ we have $\prm=3,$
hence for the second closed geodesic
$c_2$ on $S^{2m-1}:
\ind(c_2)\le 12m-10.$
The proof of Theorem~\ref{thm:length} is given in
Section~\ref{sec:proofs}.
We use the computation of the
homomorphism in homology induced by the
projection of the free
loop space $\Lambda S^n$ onto the quotient
space $\Lambda S^n/S^1$ as
given in Lemma~\ref{lem:hom-projection} 
for $n=2m-1.$ An analogous result is not
available for even dimension $n.$
% % % % % % % % % % % % % % % % % % % %
Recall that $f_0$ is the Finsler metric defined
by the standard Riemannian metric of constant
sectional curvature $1.$ 
There is a one-parameter
family $f_{\mu},
\mu \in [0,1)$ of Finsler metrics on $S^n$
starting at the standard metric $f_0$ with the
following properties: For 
every irrational $\mu$ the metric
is non-reversible and bumpy and carries exactly
$2m$ geometrically distinct closed geodesics.
For $n=2m-1$ of these closed geodesics
the index is at most $6(m-1)$
but the index of one of these closed geodesics can
be arbitrarily large. 
This example is explained 
in detail in Section~\ref{sec:katok},
these metrics were first studied by Katok,
cf.~\cite{Zi}.

The set of metrics satisfying the
assumptions of Theorem~\ref{thm:length} contains 
an open and
dense subset. This follows from
the following
%%%%%%%%%%%%%%%%%%%%%%%%%%%%%%%%%%%%%%%%%%%%%%%%%%%%%%%%%%%%%%%%%%%%%%
\begin{theorem}
	\label{thm:open-and-dense}
	Let $M$ be a compact manifold
	endowed with a Finsler metric $f_0.$
	For an arbitrary non-reversible Finsler metric
	$f$ the distortion $D=D(f)$ is the smallest
	positive number satisfying Equation~\eqref{eq:dil}
	for all tangent vectors. 
	For a positive number 
	$L$ let $\mathcal{F}_1(L)$ be the set
	of Finsler metrics $f$ on $M$ 
	all of whose closed geodesics of length
	$\le D^3(f) L$ are non-degenerate. 
	Then
	$\mathcal{F}_1(L)$ is an open and dense
	subset of the
	space $\mathcal{F}(M)$
	of all Finsler metrics on $M$
	with respect to the (strong) $C^r$-topology
	for $r\ge 4.$	
\end{theorem}
%%%%%%%%%%%%%%%%%%%%%%%%%%%%%%%%%%%%%%%%%%%%%%%%%%%%%%%%
We give the proof in
Section~\ref{sec:generic}.
The essential ingredient is the
\emph{bumpy metrics theorem}
for Finsler metrics, cf. \cite[Thm.4]{RT2020}.

Using Theorem~\ref{thm:open-and-dense} we obtain
from Theorem~\ref{thm:length} the following
%%%%%%%%%%%%%%%%%%%%%%%%%%%%%%%%%%%%%%%%%%%%%%%%%%%%%%%%
\begin{corollary}
	\label{cor:open-and-dense}	
Let $\prm$ be the smallest prime number which
is neither a divisor of $(m-1)$ nor of $m.$
Then there is an open and dense subset of non-reversible Finsler
metrics on the sphere $S^n$ of odd dimension $n=2m-1\ge 3$
carrying two geometrically distinct closed geodesics
with index $\le 4\prm (m-1)+2.$	
\end{corollary}
%%%%%%%%%%%%%%%%%%%%%%%%%%%%%%%%%%%%%%%%%%%%%%%%%%%%%%%%
%%%%%%%%%%%%%%%%%%%%%%%%%%%%%%%%%%%%%%%%%%%%
\section*{Acknowledgement}
%%%%%%%%%%%%%%%%%%%%%%%%%%%%%%%%%%%%%%%%%%%%
I am grateful to Nancy Hingston for helpful discussions
about the topic of the paper. And the suggestions and
comments of the anonymous referee helped a lot to
improve the paper.
%%%%%%%%%%%%%%%%%%%%%%%%%%%%%%%%%%%%%%%%%%%%
%%%%%%%%%%%%%%%%%%%%%%%%%%%%%%%%%%%%%%%%%%%%%%%%%%%%%%%%
%%%%%%%%%%%%%%%%%%%%%%%%%%%%%%%%%%%%%%%%%%%%%%%%%%%%%%%%%%%%%
%%%%%%%%%%%%%%%%%%%%%%%%%%%%%%%%%%%%%%%%%%%%%%%%%%%%%%%%%%%%%%
\section{Homology of the free loop space}
%%%%%%%%%%%%%%%%%%%%%%%%%%%%%%%%%%%%%%%%%%%%%%%%%%%%%%%%%%%%%
%%%%%%%%%%%%%%%%%%%%%%%%%%%%%%%%%%%%%%%%%%%%%%%%%%%%%%%%%%%%%
Closed geodesics  on $S^n$
with a Finsler metric $f $ are the critical points of the
functional
$$F:\Lambda \sn\longrightarrow \R\,;\,
F(\sigma):=\left( \int_0^1 f^2(\sigma'(t))\, dt\right)^{1/2}\,,$$
cf.~\cite[Sec.1]{HR} and~\cite[ch. 1]{Ra89}.
We denote by $\Lambda=\Lambda \sn$ the free loop space,
i.e. the space of $H^1$-maps 
$\sigma: S^1=\mathbb{R}/\mathbb{Z} \rightarrow \sn.$
The function $F$ is 
up to a factor $1/2$
the square root of the energy functional
$E(\sigma)=1/2 \int_0^1 f^2(\sigma'(t))\,dt.$
The functional $F$ agrees with the length functional
$l(\gamma)=\int_0^1 f\left(\gamma'(t)\right)\,dt$
on loops parametrized proportional to arc length.
The
\emph{Morse index} $\ind (c)$ is the maximal
dimension of a subspace
of the tangent space $T_c\Lambda S^n$
on which the hessian
$d^2F_c$ is negative definite,
cf. for example~\cite[ch. 1]{Ra89}.
For a closed geodesic $c$ the iterates $c^k, k\ge 1$ with 
$ c^k(t)=c(kt)$ are
closed geodesics, too. These closed geodesics are 
\emph{geometrically equivalent.} Note that in general
the curve $c^{-1}$ with opposite orientation,
i.e. $c^{-1}(t)=c(-t),$ is not
a closed geodesic since the metric is assumed
to be non-reversible.

For $f=f_0$ we use the following notation:
\begin{equation*}
	F_0(\sigma)=:\left( \int_0^1 f_0^2(\sigma'(t))\, dt\right)^{1/2}\,;
	\,
	l_0(\sigma)=\int_0^1 f_0\left(\sigma'(t)\right)\,dt	\,.
\end{equation*}
%%%%%%%%%%%%%%%%%%%%%%%%%%%%%%%%%%%%%%%%%%%%%%%%%%%%%%%%%%%%%%%%
For the sublevel
sets of the functional $F$
we use the following notation:
$\la^{R}=\{\sigma\in \la\,;\,
F(\sigma)\le R\}.$ 
The free loop space
$\Lambda$ carries a canonical 
$S^1$-action by linear reparame\-tri\-zation
of the curves, i.e. shift of the
initial point. We use the following
notation for quotient spaces 
with respect to the $\sone$-action
and its
sublevel spaces:
$\ola=\la /S^1$ and
$\ola^R=\{\sigma\in \ola\,;\,
F(\sigma)\le R\}.$
For the sublevel sets with respect to the functional $F_0$
we use the following notation:
$\la_0^{R}=\{\sigma\in \la\,;\,
F_0(\sigma)\le R\},$ and 
$\ola_0^R=\{\sigma\in \ola\,;\,
F_0(\sigma)\le R\}.$
%%%%%%%%%%%%%%%%%%%%%%%%%%%%%%%%%%
%%%%%%%%%%%%%%%%%%%%%%%%%%%%%%%%%%%%%%%%%%%%%%%%%%%%%%%%%%%%%%
%%%%%%%%%%%%%%%%%%%%%%%%%%%%%%%%%%%%%%%%%%%%%%%%%%%%%%
The set of prime closed geodesics of positive
length 
of the standard metric $f_0$
equals the subset 
$BS^n\subset \la S^n$ of great circles
which can be identified with the unit tangent
bundle $T^1S^n.$ Then the set of closed geodesics
equals the union 
$\bigcup_{j\ge 1} B^j.$ Here
$B^j:=\{c_0^j,c_0 \in BS^n\}$ is the set of
$j$-fold covered great circles, i.e. great circles
$c_0$ parametrized
proportional to arc length with $l_0(c_0^j)=jl_0(c_0)
=2\pi j.$ The functional 
$F_0:\la S^n \longrightarrow \R$ is a 
Morse-Bott function, i.e. the subsets
$B^j$ are non-degenerate critical submanifolds.
This follows since the dimension of
the kernel of the hessian of a great circle
equals the dimension $2n-1$ of the manifold
$BS^n=T^1S^n.$
For $n=2m-1\ge 3$ we have 
\beqn
H_j\left(T^1S^{2m-1};\z\right)\cong
\left\{
\begin{array}{ccc}
	\z&;&j=0,2m-2,2m-1,4m-3\\
	0&;& \text{otherwise}
\end{array}
\right.\,.
\eeqn
If $\nu_k:N_k \longrightarrow B^k$ is the
negative normal bundle of
the critical submanifold $B^k$
of dimension $\ind(c^k)=(4k-2)(m-1)$
with the associated disc bundle
$\nu_k:DN_k\longrightarrow B^k,$ resp.
sphere bundle 
$SN_k\longrightarrow B^k,$
then the generalized Morse lemma implies
% % % % % % % % % % % % % % % % % % % % % % %
\beqn
H_j(
\la_0^{2\pi k} ,
\la_0^{2\pi (k-1)};\z
)
\cong
H_{j}
(DN_k,SN_k;\z)\,,
\eeqn
% % % % % % % % % % % % % % % % % % % % % % % % %
cf.~\cite[\S 4]{RaMZ}.
The negative normal bundle $\nu_k$ is oriented
for all $k,$ since
$\ind(c_0^2)-\ind(c_0)=4(m-1)$
is even resp. $\gamma_{c_0}=1$ for a great circle
$c_0$,
cf. \cite[Prop.2.2]{Ra89}.
Hence the 
Thom-isomorphism implies
% % % % % % % % % % % % % % % % % % % %
\beq
\label{eq:hlambdasm1}
H_j(
\la_0^{2\pi k},
\la_0^{2\pi (k-1)};\z
)
\cong
H_{j-(4k-2)(m-1)}
(T^1S^{2m-1};\z
)\,.
\eeq
% % % % % % % % % % % % % % % % % % % % %
The functional 
$F_0$ is \emph{perfect,} i.e.
\beq
\label{eq:hlambdaglobal}
H_j(
\la ,
\la^0 ;\z
)
\cong
\bigoplus_{k\ge1}
H_j(
\la_0^{2\pi k},
\la_0^{2\pi (k-1)};\z
)
\eeq
which follows for $m\ge 2$
from the long exact
homology sequence. Hence
%%%%%%%%%%%%%%%%%%%%%%%%%%%%%%%%%%%%%%%%%%%%%%%%%%%%%%%%%
\begin{equation}
	\label{eq:hlambdasm}
	H_j\left(\la,\la^0 ;\z\right)
	\cong
	\left\{
	\begin{array}{lll}
		\z &;& j=2r(m-1); r\ge 1\\
		\z &;& j=2r(m-1)+1, r\ge 2\\
		0 &;& \
		\textrm{otherwise}
	\end{array} 
	\right.	,
\end{equation}
%%%%%%%%%%%%%%%%%%%%%%%%%%%%%%%%%%%%%%%%%%%%%%%%%%%%%%%%%
and the homomorphism
%%%%%%%%%%%%%%%%%%%%%%%%%%%%%%%%%%%%%%%%%%%%%%%%%%%%%%%%
\begin{equation}
	\label{eq:hom-iso}
	H_j\left(\Lambda_0^{2\pi k},\Lambda^0;\q\right)	
	\longrightarrow
	H_j\left(
	\Lambda,\Lambda^0,\q
	\right)
\end{equation}
%%%%%%%%%%%%%%%%%%%%%%%%%%%%%%%%%%%%%%%%%%%%%%%%%%%%%%%%%
induced by the inclusion is an isomorphism for all
$k\ge 1$ and
$j<i(k+1)=(4k+2)(m-1).$ This follows since
$i(k+1)=\ind(c_0^{k+1})=
i(k)+4(m-1).$

The quotient space
$T^1S^n/S^1$ of unparametrized oriented
great circles can be identified with
the Grassmannian $\widetilde{G}(2,2m-2)$
of oriented two-dimensional linear subspaces of
$\R^{2m}.$

The equivariant Morse lemma implies
% % % % % % % % % % % % % % % % % % % % %
\beqn
H_j(
\ola_0^{2\pi k},
\ola_0^{2\pi (k-1)};\z
)
\cong
H_{j}
\left(\overline{DN}_k,
\overline{SN}_k;\z
\right)\,,
\eeqn
% % % % % % % % % % % % % % % % % % % %
cf. \cite[\S4]{RaMZ}.
Here the quotient bundle
$\nu_k:\overline{DN}_k\longrightarrow \overline{B}_k$
resp.
$\nu_k:\overline{SN}_k\longrightarrow \overline{B}_k$
is a bundle with fibre
$D^{i(k)}/\z_k$ resp.
$S^{i(k)-1}/\z_k.$ Here
$i(k)=\ind (c_0^k)=(4k-2)(m-1)$ is the Morse index of
a $k$-fold covered great circle $c_0^k$ as 
a closed
geodesic of the standard metric $f_0.$ 
Then
% % % % % % % % % % % % % % % %
\beqn
H_*\left(
D^{i(k)}/\z_k,
S^{i(k)-1}/\z_k;
\q
\right)
\cong
H_*\left(
D^{i(k)},
S^{i(k)-1};
\q
\right)
\eeqn
% % % % % % % % % % % % % % % %
and the Thom isomorphism implies
% % % % % % % % % % % % % % % % % %
\beqn
H_*(\ola_0^{2\pi k}
,\ola_0^{ 2 \pi (k-1)}
; \mathbb{Q})
\cong
H_{*-i(k)}(\widetilde{G}\left(2,2m-2\right),
\mathbb{Q})\,.
\eeqn
% % % % % % % % % % % % % % % % % % % % % %
Non-trivial homology only occurs in even
dimensions since
% % % % % % % % % % % % % % % % % % % % % % % %
\beq
\label{eq:homteinssm}
H_j(
\widetilde{G}(2,2m-2)
;\z
)\cong
\left\{
\begin{array}{ccc}
	\z&;&j=0,4m-4\\
	\z\oplus \z&;& j=2m-2\\
	0 &;& \text{otherwise}
\end{array}
\right.
\,.
\eeq
% % % % % % % % % % % % % % % % % % % %
This follows from the Gysin sequence
of the $S^1$-bundle
$T^1\sm \longrightarrow \widetilde{G}(2,2m),$
cf. \cite[Beweis Satz 4.9]{RaDiss}.
Hence we obtain:
\beqn
H_{*}(\ola,\ola^0;
\mathbb{Q})
\cong \bigoplus_{k\ge 1} 
H_{*}(\ola_0^{2\pi k},
\ola_0^{2\pi (k-1)}
;\mathbb{Q})\,,
\eeqn
% % % % % % % % % % % % % % % % % % % % % %
which implies
% % % % % % % % % % % % % % % % % % % % % %
\begin{equation}
	\label{eq:homlasone}
	H_j(\ola,\ola^0;\q)
	\cong
	\left\{
	\begin{array}{lll}
		\q &;& j\ge 2(m-1), j \text{ even }, j\not=2k(m-1), k\ge 2\\
		\q \oplus \q &;&j=2k(m-1), k\ge 2 \\
		0 &;& \
		\textrm{otherwise}
	\end{array} 
	\right.	,
\end{equation}
% % % % % % % % % % % % % % % % % % % % % % %
\cite[Rem.2.5(a)]{Ra89}.
Therefore the functional
$F_0: \ola \longrightarrow \R$ can be seen as a perfect
Morse Bott function
for
rational coefficients, too.
In particular the homomorphism
%%%%%%%%%%%%%%%%%%%%%%%%%%%%%%%%%%%%%%%%%%%%%%%%%%
\beqn
%\label{eq:inclusion3}
i_*\,:\,H_j(\ola_0^{ 2 \pi  k},
\ola^0 ;\mathbb{Q})
\longrightarrow 
H_j(\ola,
\ola^0;\mathbb{Q})
\eeqn
%%%%%%%%%%%%%%%%%%%%%%%%%%%%%%%%%%%%%%%%%%%%%%%%%%%
induced by the inclusion is an isomorphism for all
$k\ge 1$ and $j < i(k+1)= (4k+2)(m-1)\,.$
This follows since $i(k+1)=\ind (c_0^{k+1})=
i(k)+4(m-1).$
%%%%%%%%%%%%%%%%%%%%%%%%%%%%%%%%%%%%%%%%%%%%%%%%%%%%%
%%%%%%%%%%%%%%%%%%%%%%%%%%%%%%%%%%%%%%%%%%%%%%%%%%%%%
\begin{lemma}
	\label{lem:hom-projection}
	$n=2m-1, m\ge 2.$
	Let $a_k \in H_{4k(m-1)}(\Lambda  ,\Lambda^0;\z)
	\cong \z, k\ge 1$ be a generator.
	Then the canonical projection $\rho: \left(\Lambda,
	\Lambda^0 \right) \longrightarrow
	(\ola ,\ola	^0)$
	induces an injective homomorphism
	\begin{equation*}
		\rho_*:
		H_{4k(m-1)} (\Lambda,
		\Lambda^0;\z)
		\cong \z \longrightarrow
		H_{4k(m-1)}
		(\ola,\ola^0;\z)	
	\end{equation*}
	with $\rho_*(a_k)=k \tilde{a}_k\not=0$
	and $\tilde{a}_k$ is not a torsion element.
\end{lemma}
%%%%%%%%%%%%%%%%%%%%%%%%%%%%%%%%%%%%%%%%%%%%%%%%%%%%%%%%%%
%%%%%%%%%%%%%%%%%%%%%%%%%%%%%%%%%%%%%%%%%%%%%%%%%%%%%%%%%%
\begin{proof}
	The projection 
	$\rho: \la\longrightarrow \ola$ 
	induces the homomorphism
	%%%%%%%%%%%%%%%%%%%%%%%%%%%%%%%%%%%%%%%%%%%%%%%%%%
	\beqn
	\rho_*:
	H_{4k(m-1)}
	(
	\la, \la^0;\z
	)\cong \z
	\longrightarrow
	H_{4k(m-1)}
	(
	\ola, \ola^0;\z
	)\,.
	\eeqn
	%%%%%%%%%%%%%%%%%%%%%%%%%%%%%%%%%%%%%%%%%%%%%%
	The homomorphism
	%%%%%%%%%%%%%%%%%%%%%%%%%%%%%%%%%%%%%%%%%%%%%%%%
	\beqn
	\rho_*:
	H_{4k(m-1)}
	(
	\la_0^{2\pi k},
	\la_0^{2\pi (k-1)};
	\z
	)
	\longrightarrow
	H_{4k(m-1)}
	(
	\ola_0^{2\pi k},
	\ola_0^{2\pi (k-1)};
	\z
	)
	\eeqn
	%%%%%%%%%%%%%%%%%%%%%%%%%%%%%%%%%%%%%%%%%%%%%%%%%%%%
	can be expressed by the homomorphism
	%%%%%%%%%%%%%%%%%%%%%%%%%%%%%%%%%%%%%%%%%%%%%%%%%%%%%
	\beqn
	\rho_*:
	H_{4k(m-1)}
	\left(
	DN_k,SN_k;\z
	\right)\cong \z
	\longrightarrow
	H_{4k(m-1)}
	\left(
	\overline{DN}_k,
	\overline{SN}_k;\z
	\right)
	\eeqn
	%%%%%%%%%%%%%%%%%%%%%%%%%%%%%%%%%%%%%%%%%%%%%%%%%%%%%%%%%%%%
	which is a multiplication with the
	number $k,$ i.e.  for a generator $a_k'$ with $0\not=a_k'
	\in H_{4k(m-1)}(DN_k,SN_k;\z)
	\cong \z$ we have
	$\rho_*(a_k')=k s\overline{a}_k$
	for a generator\\
	$\overline{a}_k \in 
	H_{4k(m-1)}
	\left(
	\overline{DN}_k,
	\overline{SN}_k;\z
	\right)$ and an integer $s>0.$
	This follows since
	the homomorphism
	%%%%%%%%%%%%%%%%%%%%%%%%%%%%%%%%%%%%%%%%%%%%%%%%%%%%%%%%%%%%
	\begin{eqnarray*}
		H_{4k(m-1)}(D^{4k(m-1)},
		S^{4k(m-1)-1};\z)\cong \z&\\
		\longrightarrow
		&H_{4k(m-1)}(D^{4k(m-1)}/\z_k,
		S^{4k(m-1)-1}/\z_k;\z)
		\cong \z
	\end{eqnarray*}
	%%%%%%%%%%%%%%%%%%%%%%%%%%%%%%%%%%%%%%%%%%%%%%%%%%%%%%%%%%%%
	induced by the canonical projection is a
	multiplication by $k.$
	This follows since the isometric
	$\z_k$-action on the disc
	$D^{4k(m-1)}$
	is free on an open and dense subset,
	which we see as follows:
	For any divisor $d|k, d<k$ we have
	the following inequality
	for the indices of coverings
	$c^k_0$ of a great circle $c_0:$
	$\ind(c^d_0)<\ind(c^k_0).$
	Actually one can show $s=1,$ 
	i.e. $\rho_*(a_k)=k\tilde{a}_k.$
	This follows
	from the Gysin sequence of the
	$S^1$-bundle
	$T^1\sm\longrightarrow \widetilde{G}(2,2m-2)$
	and 
	Equation~\eqref{eq:homteinssm}.
\end{proof}
%%%%%%%%%%%%%%%%%%%%%%%%%%%%%%%
\begin{remark}
	The $S^1$-action on $\Lambda$
	induces the homomorphism
	\label{rem:transfer}
	\beqn
	\Delta:
	H_{4k(m-1)}
	(
	\la,
	\la^0;\z 
	)\cong \z \cdot a_k
	\longrightarrow
	H_{4k(m-1)+1}
	(
	\la,
	\la^0;\z
	)\,,
	\eeqn
	cf. \cite[(17.1)]{GH}.
	The homomorphism is used
	to define the 
	\emph{Batalin Vilkovisky algebra,}
	cf. \cite[Thm.5.4]{CS}.
	% % % % % % % % % % % % % % % % % % % % % % %
	It can be expressed as composition 
	$\Delta=\tau \circ \rho_*$
	of the 
	homomorphism 
	% % % % % % % % % % % % % % % % % % % % % % %
	\beqn
	\rho_*:
	H_{4k(m-1)}
	(
	\la,
	\la^0;\z
	)\cong \z \cdot a_k
	\longrightarrow
	H_{4k(m-1)}
	(
	\ola,
	\ola^0;\z
	)
	\eeqn
	%%%%%%%%%%%%%%%%%%%%%%%%%%%%%%%%%%%%%%%%%%%%%%%%%%
	induced by the canonical projection and
	the \emph{transfer map}
	%%%%%%%%%%%%%%%%%%%%%%%%%%%%%%%%%%%%%%%%%%%%%%%%%%%
	\beqn
	\tau:
	H_{4k(m-1)}
	(
	\ola,
	\ola^0;\z
	)
	\longrightarrow
	H_{4k(m-1)+1}
	(
	\la,
	\la^0 ;\z
	)\cong \z \cdot  \tilde{a}_k\,.
	\eeqn
	%%%%%%%%%%%%%%%%%%%%%%%%%%%%%%%%%%%%%%%%%%%%%%%%%%%%
	Hence we obtain
	$\Delta (a_k)=\tau (\rho_*(a_k))=
	k \tilde{s} \tilde{a}_k$ for a positive integer $\tilde{s}$
	and a generator 
	$\tilde{a}_k \in
		H_{4k(m-1)+1}(
	\la,\la^0 ;\z).$
	The homomorphism can be computed,
	cf. \cite[Satz~4.13]{RaDiss} resp.
	\cite[Lem.6.2]{HR}, it follows that
	$\tilde{s}=2.$
\end{remark}
%%%%%%%%%%%%%%%%%%%%%%%%%%%%%%%%%%%%%%%%%%%%%%%%%%%%%%%%%%%%%%%
	%%%%%%%%%%%%%%%%%%%%%%%%%%%%%%%%%%%%%%%%%%%%%%%%%%%%%%%%%%%%%%%%%%%%%%%
	\begin{remark}
	Since $c$ is prime and since for all divisors $q$ of $r$
	with $q<r$ the inequality $\ind(c^q)<\ind(c^r)$ holds, we can
	conclude that for $r\ge 1$
	the following holds:
	There are generators 
	%%%%%%%%%%%%%%%%%%%%%%%%%%%%%%%%%%%%%%%%%%%%%%%%%%%%%%%%%%%%
	\beq
	\label{eq:sr}
	s_r, t_r\in H_*(\la^{rL},\la^{(r-1)L} ;\z)\,;\,
	S_r \in H_*(\ola^{rL},\ola^{(r-1)L};\z)
	\eeq
	%%%%%%%%%%%%%%%%%%%%%%%%%%%%%%%%%%%%%%%%%%%%%%%%%%%%%%%%%%%
	with $\deg(s_r)=\deg(S_r)=\deg(t_r)-1=\ind(c^r)=j$
	such that the induced projection
	%%%%%%%%%%%%%%%%%%%%%%%%%%%%%%%%%%%%%%%%%%%%%%%%%%%%%%%%%%%%%%%%%%%%
	\beq
	\rho_*: H_j(\la^{rL},\la^{(r-1)L};
	\z)\cong \z\cdot s_r
	\longrightarrow
	H_{j}(\ola^{rL},\ola^{(r-1)L};
	\z)\cong \z \cdot S_r
	\eeq
	%%%%%%%%%%%%%%%%%%%%%%%%%%%%%%%%%%%%%%%%%%%%%%%%%%%%%%%%%%%%%%%%%%%%%%
	satisfies
	%%%%%%%%%%%%%%%%%%%%%%%%%%%%%%%%%%%%%%%%%%%%%%%%%%%%%%%%%%%%%%%%%%%%%%
	\beq
	\label{eq:psr}
	\rho_*(s_r)= r \cdot S_r \,,
	\eeq
	%%%%%%%%%%%%%%%%%%%%%%%%%%%%%%%%%%%%%%%%%%%%
	cf. \cite[\S3]{RaMZ}.
	This will be crucial in the Proof
	of Theorem~\ref{thm:length} given in Section~\ref{sec:proofs}.
	For the transfer homomorphism 
	%%%%%%%%%%%%%%%%%%%%%%%%%%%%%%%%%%%%%%%%%%%%%%%%%%%%%%%%%%%%%%%%%%%%%%%%%%%%%
	\beqn
	\Delta: H_j(\Lambda^{rL},\Lambda^{(r-1)L};
	\z)\cong\z \cdot s_r
	\longrightarrow
	H_{j+1}(\ola^{rL},\ola^{(r-1)L};
	\z)\cong\z \cdot t_r
	\eeqn
	%%%%%%%%%%%%%%%%%%%%%%%%%%%%%%%%%%%%%%%%%%%%%%%%%%%%%%%%%%%%%%%%%%%%%%%%%%%%%
	one obtains $\Delta (s_r)= r\cdot t_r\,.$
	%%%%%%%%%%%%%%%%%%%%%%%%%%%%%%%%%%%%%%%%%%%%%%%%%%%%%%%%%%%%%%%%%%%%%%%%%%%%%%
\end{remark}
%%%%%%%%%%%%%%%%%%%%%%%%%%%%%%%%%%%%%%%%%%%%%%%%%%%%%%%%%%%%%%%%%
%%%%%%%%%%%%%%%%%%%%%%%%%%%%%%%%%%%%%%%%%%%%%%%%%%%%%%%%
\section{Morse theory for a metric with only
	one closed geodesic}
%%%%%%%%%%%%%%%%%%%%%%%%%%%%%%%%%%%%%%%%%%%%%%%%%%%%%%%%
%%%%%%%%%%%%%%%%%%%%%%%%%%%%%%%%%%%%%%%%%%%%%%%%%%%%%%%%
%%%%%%%%%%%%%%%%%%%%%%%%%%%%%%%%%%%%%%%%%%%%%%%%%%%%%%%%
In this section we study non-reversible
Finsler metrics on $S^{2m-1}$ for which all closed
geodesics with length $\le 2\pi p_m D^3(f)$ are geometrically
equivalent to the closed geodesic $c$
of length $L=l(c).$ We will show that
this assumption determines the sequence
$\ind(c^r), rL \le 2\pi p_m D^3$ completely.
%%%%%%%%%%%%%%%%%%%%%%%%%%%%%%%%%%%%%%%%%%%%%%%%%%%%%%%%
\begin{lemma}
\label{lem:one-closed-geodesic}	
Let $f$ be a non-reversible Finsler metric on
the sphere $S^n,n\ge 2$ with distortion 
$D=D(f).$	We assume that all closed geodesics
with length $\le 2\pi D$ 
are non-degenerate. 
Then there exists a prime closed
geodesic $c$ 
whose length satisfies
$L:=l(c) \le 2\pi D$ and with
$\ind (c)\le n-1.$
\end{lemma}
%%%%%%%%%%%%%%%%%%%%%%%%%%%%%%%%%%%%%%%%%%%%%%%%%%%%%%%%%%%%%%%%
\begin{proof}
Equation~\eqref{eq:dil}
implies that
$\Lambda_0^{2\pi}\subset \Lambda^{2\pi D}.$
Since 
\begin{equation*}
H_{n-1}\left(
\Lambda_0^{2\pi},\Lambda^0;\q
\right)	
\longrightarrow 
H_{n-1}
\left(
\Lambda,\Lambda^0;\q
\right)\cong \q	
\end{equation*}
is an isomorphism, cf.
Equation~\eqref{eq:hom-iso},
we conclude that the homomorphism 
\begin{equation*}
H_{n-1}\left(	
\Lambda^{2\pi D}, \Lambda^0;\q
\right)\longrightarrow 
H_{n-1}\left(
\Lambda , \Lambda^0;\q
\right)\cong \q
\end{equation*}
is surjective, i.e.
$\dim H_{n-1}\left(\Lambda^{2\pi D},\Lambda^0;\q
\right)	\ge 1.$
It follows from the Morse
inequalities for the space
$\Lambda^{2\pi D}$
that there is a closed geodesic
$c$ with length $l(c)\le 2\pi D$
and index $\ind(c)\le n-1,$
cf.~\cite[Sec.2]{Ra89}.
\end{proof}
%%%%%%%%%%%%%%%%%%%%%%%%%%%%%%%%%%%%%%%%%%%%%%%%%%%%%%%%%%%%%%%%%
%Lemma~\ref{lem:one-closed-geodesic} implies that %there
%is at least one closed geodesic with length
%$\le 2\pi D$ and index $\le n-1.$ 
We later use the following 
%%%%%%%%%%%%%%%%%%%%%%%%%%%%%%%%%%%%%%%%%%%%%%%%%%%%%%%%%%%%%%%
\begin{assumption}
	\label{assumption}
	For $m\ge 2$ let $\prm$
	be the smallest prime number which does
	not divide $(m-1)$ nor $m,$
	cf. Lemma~\ref{lem:pm}.
	Given a non-reversible Finsler metric $f$ on a
	sphere of dimension $n=2m-1\ge 3$
	with distortion $D=D(f)$
	we assume that all closed geodesics $\gamma$ with
	$L(\gamma)\le \lm:=2\pi \prm D^3$ are non-degenerate
	and that all closed geodesics with
	length $\le \lm=2\pi \prm D^3$
	and index $\le 4\prm(m-1)+2$ are geometrically
	equivalent.
\end{assumption}
%%%%%%%%%%%%%%%%%%%%%%%%%%%%%%%%%%%%%%%%%%%%%%%%%%%%%%%%%%%%%%%%
	Hence we conclude from
	Lemma~\ref{lem:one-closed-geodesic}
	that there is a prime closed geodesic $c$ 
	such that
	every closed geodesic $\gamma$ with 
	$l(\gamma)\le \lm= 2\pi \prm D^3$
	and $\ind(\gamma)\le
	4\prm (m-1)+2$ is up to the
	choice of the initial point a covering 
	of the closed geodesic $c,$ i.e.
	there is a positive integer $r\ge 1$ 
	and an element $z\in \sone=\R/\z=[0,1]/\{0,1\}$ such
	that $\gamma=z.c^r.$
	Here $z.c(t)=c(t+z)$ defines the 
	canonical $\sone$-action on the free
	loop space $\Lambda=\Lambda \sn$ leaving the functional
	$F$ invariant.
%%%%%%%%%%%%%%%%%%%%%%%%%%%%%%%%%%%%%%%%%%%%%%%%%%%%%%%%%
\begin{lemma}
	\label{lem:pm}
	For $m\ge 2$ denote by $\prm$ the smallest prime number,
	which is neither a divisor of $(m-1)$ nor of $m.$
	Then $p_2=3, p_3=5,$ and for $m\ge 4: 3 \le \prm \le m+1.$
\end{lemma}
%%%%%%%%%%%%%%%%%%%%%%%%%%%%%%%%%%%%%%%%%%%%%%%%%%%%%%%%%%%%%%%%%%%%%%%
\begin{proof}	
	For $m \le 5$ we have: 	
	$p_2=3, p_3=p_4=5.$
	Assume $m \ge 5.$ 
	If $m \equiv 2\pmod{3}$ then $\prm=3.$
	If $m-2\not=2^s$ for some $s$ choose a prime factor $q\ge 3$ of $m-2\ge 4.$ If $m-2=2^s,$ choose a prime factor
	$q\ge 3$ of $m+1.$ Then $p_m\le q \le m+1,$
	and hence 	$3 \le p_m\le m+2$ for
	all $m\ge 2.$ 	
\end{proof}
%%%%%%%%%%%%%%%%%%%%%%%%%%%%%%%%%%%%%%%%%%%%%%%%%%%%%%%%%%%%%%%%%%%%%%%
The invariant 
$\gamma_c \in \left\{\pm 1/2,\pm 1\right\}$
of a prime closed geodesic is defined as
follows:
$\gamma_c=\pm 1$ if and only if $\ind(c^2)-\ind(c)$ is even and
$\gamma_c>0$ if and only if $\ind (c) $ is even, cf.~\cite[Def. 1.6]{Ra89}.
%%%%%%%%%%%%%%%%%%%%%%%%%%%%%%%%%%%%%%%%%%%%%%%%%%%%%%%%%%%%%%%%%%%%%%
%%%%%%%%%%%%%%%%%%%%%%%%%%%%%%%%%%%%%%%%%%%%%%%%%%%%%%%%%%%%%%%%%%%	
\begin{lemma}
	\label{lem:morse}
	Let Assumption~\ref{assumption}
	be satisfied, i.e. there exists a prime
	closed geodesic $c$ with $L=l(c)\le 2\pi D$ such
	that all closed geodesics $\gamma$
	with length $l(\gamma)\le 2\pi p_m D^3$
	and index 	$\ind(\gamma)\le
	4\prm (m-1)+2$ 
	are geometrically equivalent to $c,$
	cf. Lemma~\ref{lem:one-closed-geodesic}.
	
	We use the following
	notation for Betti numbers of the 
	quotients
	$\ola^{2\pi D \prm}$ and $\ola^{2\pi D^3 \prm}$
	of the sublevel sets
	by the canonical $S^1$-action: 
	\begin{equation*}
	\overline{\beta}_j:=
	\dim H_j(\ola^{2\pi D \prm},
	\ola^0;\q)\,,\,
	\overline{\beta}_j^*:=\dim H_j(\ola^{2\pi D^3\prm},
	\ola^0;\q)\,.
	\end{equation*}
Let $L_1=2\pi p_m D, L_3=2\pi p_m D^3=L_1 D^2$ and
	let
	%%%%%%%%%%%%%%%%%%%%%%%%%%%%%%%%%%%%%%%%%%%%%
	\begin{eqnarray*}
	v_j&:=&\#\{1\le r\le \leins/L;\, \ind(c^r)=j, r\equiv 1\pmod{2}
	\mbox{ or } \gamma_c=\pm 1\}\\
	v_j^*&:=&
	\#\{1\le r\le \lm/L;\, \ind(c^r)=j, r\equiv 1\pmod{2}
	\mbox{ or } \gamma_c=\pm 1\}\,.		
	\end{eqnarray*}
%%%%%%%%%%%%%%%%%%%%%%%%%%%%%%%%%%%%%%%%%%%%%%%%%%%%%%%%
Then for all even $j \le 4\prm (m-1)+2:$
%%%%%%%%%%%%%%%%%%%%%%%%%%%%%%%%%%%%%%%%%%%%%%%%%%%%%%%%%%%%%%%%
\begin{equation*}
		\overline{\beta}_j = v_j\,;\, 
	\overline{\beta}_j^*=v_j^*
\end{equation*}
%%%%%%%%%%%%%%%%%%%%%%%%%%%%%%%%%%%%%%%%%%%%%%%%%%%%%%%%%%%%%%%%
and $\overline{\beta}_j=\overline{\beta}_j^*=0$ for all odd 
$j \le 4\prm (m-1)+2.$
\end{lemma}	
%%%%%%%%%%%%%%%%%%%%%%%%%%%%%%%%%%%%%%%%%%%%%%%%%%%%%%%%%%%%%%%%
\begin{proof}
	We conclude from \cite[Sec.2]{Ra89}
	and \cite[Def.1.6]{Ra89}
	or \cite[Sec.2]{Ra10}:\\
	Let $v_j(c^r)=b_j(\ola^{rL},
	\ola^{(r-1)L}\,;\,
	\mathbb{Q}).$
	If $l(c^r)=rl(c)\le \lm$ we conclude from
	Assumption~\ref{assumption}
	for all $j\le 4\prm(m-1)+2:$
	$v_j(c^r)  \in \{0,1\}$ with $v_j(c^r)=1$ if and only if
	$j=\ind(c^r)$ and $r$ is odd or $\ind(c^2)\equiv \ind(c)\pmod{2}.$
	It follows that for $1\le r\le \lm/L:$
	\beq
	\label{eq:fet}
	v_j(c^r)=1 \Rightarrow j=\ind(c^r)\equiv \ind(c)\pmod{2}\,.
	\eeq
	%%%%%%%%%%%%%%%%%%%%%%%%%%
	The \emph{Morse inequalities} for the functional $F$ on the space
	$\ola^{\leins}=\ola^{\leins}\sn$ resp.
	$\ola^{\lm}=\ola^{\lm}\sn$ give a relation between the number of
	(homologically visible) critical points 
	$v_j,$ resp.
	$v_j^*$ with index $j$ and
	length 
	$l\le \leins,$ resp. $\le \lm$ with the Betti numbers 
	$\overline{\beta}_j,$ resp. $\overline{\beta}^*_j.$
	We obtain:
	%%%%%%%%%%%%%%%%%%%%%%%%%%%%%%%%%%%%%%%%%%%%%%%%%%%%%
	\beqn
	%\label{eq:morseR1}
	v_j=\overline{\beta}_j +q_j+q_{j-1}
		%%%%%%%%%%%%%%%%%%%%%%%%%%%%%%%%%%%%%%%%%%%%%%%%%%%%%%%%%
	\, \mbox{ resp. } \,
	%%%%%%%%%%%%%%%%%%%%%%%%%%%%%%%%%%%%%%%%%%%%%%%%%%%%%%%%%%
		%\label{eq:morseR}
	v^*_j=\overline{\beta}_j^* +
	q_j^*+q_{j-1}^*
	\eeqn
	%%%%%%%%%%%%%%%%%%%%%%%%%%%%%%%%%%%%%%%%%%%%%%%%%%%%%%%%%
	for a non-negative sequence 
	$q_j, j\ge 0,$ resp. $q_j^*, j\ge 0,$
	cf.~\cite[Sec.2]{Ra89}. 
	Equation~\eqref{eq:fet} implies the following for
	all $j \le 4\prm (m-1)+2,  j\equiv 1\pmod{2}$
	%%%%%%%%%%%%%%%%%%%%%%%%%%%%%%%%%%%%%%%%%%%%%%%%%%%%%%%%%
	\begin{equation}
	\label{eq:eq:vjzero}	
	v_j=v^*_j=0	
	\end{equation}
	%%%%%%%%%%%%%%%%%%%%%%%%%%%%%%%%%%%%%%%%%%%%%%%%%%%%%%%%%%
	and 
	$q_j=q^*_j=0$ for all $j.$ 
	 Here we have used that under
	the assumptions of the Lemma there is up to geometric
	equivalence only one closed geodesic of length $\le 2\pi \prm D^3,$
	and that an iterate $c^r$ can have non-trivial local
	homology in degree $j$ 
	only for even $j,$
	cf. Equation~\eqref{eq:fet}. 
	Hence	
	%%%%%%%%%%%%%%%%%%%%
		\begin{equation}
			\label{eq:morse2}
		v_j=\overline{\beta}_j;v_j^*=\overline{\beta}_j^*
	\end{equation}
for all $j \le 4 p_m(m-1)+2.$
\end{proof}
%%%%%%%%%%%%%%%%%%%%%%%%%%%%%%%%%%%%%%%%%%%%%%%%%%%%%%%%%%%%%
%%%%%%%%%%%%%%%%%%%%%%%%%%%%%%%%%%%%%%%%%%%%%%%%%%%%%%%%%%%%%%%
For a topological pair $(X,A)$ with singular
homology $H_j(X,A;\z)$
with integer coefficients let $\textrm{Tor}_j\subset H_j(X,A)$
be the torsion submodule. We denote by 
$FH_j(X,A;\z)=H_j(X,A;\z)/\textrm{Tor}_j$ the
associated free module. Then
$H_j(X,A;\q)\cong H_j(X,A;\z)\otimes \q\cong
FH_j(X,A;\z)\otimes \q.$
%%%%%%%%%%%%%%%%%%%%%%%%%%%%%%%%%%%%%%%%%%%%%%%%%%%%%%%%%%%%%%
\begin{lemma}
	\label{lem:two}
	If the Finsler metric $f$ on $\sm$ satisfies
	Assumption~\ref{assumption} and $p=\prm$
	then the homomorphism 
	%%%%%%%%%%%%%%%%%%%%%%%%%%%%%%%%%%%%%%%%%%%%%%%%%%%%%%%%%
	\beqn
	%\label{eq:inclusion1}
	H_j(\ola^{2 \pi  p
		\dil},
	\ola^0;\mathbb{Q})
	\longrightarrow 
	H_j(\ola,
	\ola^0;\mathbb{Q})
	\eeqn
	%%%%%%%%%%%%%%%%%%%%%%%%%%%%%%%%%%%%%%%%%%%%%%%%%%%%%%%%%
	induced by the inclusion is an 
	isomorphism for all $j\le 4p(m-1)+2.$
	Using the notation from Lemma~\ref{lem:morse} we obtain
	for the Betti numbers 
	$\overline{b}_j:=\dim H_j(\ola , \ola^0,\q)$
	for all $j\le 4p (m-1)+2:$
	\begin{equation}
	\overline{\beta}_j=
	\overline{b}_j\,.
	\end{equation}
\end{lemma}
%%%%%%%%%%%%%%%%%%%%%%%%%%%%%
%%%%%%%%%%%%%%%%%%%%%%%%%%%%%
\begin{proof}
From the definition of the distortion 
given in Equation~\eqref{eq:dil} we
obtain the following inclusions:
\begin{equation}
	\label{eq:D}
\Lambda_0^{2\pi p}\subset
\Lambda^{2\pi p D}\subset
\Lambda_0^{2\pi p D^2}\subset
\Lambda^{2\pi p D^3}	
\end{equation}
and
\begin{equation*}
\ola_0^{2\pi p}\subset
\ola^{2\pi p D}\subset
\ola_0^{2\pi p D^2}\subset
\ola^{2\pi p D^3}	\,.
\end{equation*}
%%%%%%%%%%%%%%%%%%%%%%%%%%%%%%%%%%%%%%%%%%%%%%%%%%%%%%%%%%%
It follows that the composition
%%%%%%%%%%%%%%%%%%%%%%%%%%%%%%%%%%%%%%%%%%%%%%%%%%%%%
\begin{equation}
	\label{eq:hom_inclusion}
H_j(\ola^{2\pi p}_0,
\ola^0;\q)
\longrightarrow
H_j(
\ola^{2\pi pD},\ola^0;\q
)
\longrightarrow
H_j
(
\ola_0^{2\pi p D^2},\ola^0;\q
)\cong  
H_j 
( 
\ola ,\ola^0;\q
)
\end{equation}
%%%%%%%%%%%%%%%%%%%%%%%%%%%%%%%%%%%%%%%%%%%%%%%%%%%%%%%%%%%
is an isomorphism for $j\le 4 p(m-1)+2,$
cf. Equation~\eqref{eq:homlasone} and the arguments below.
Therefore we conclude that the homomorphism
%%%%%%%%%%%%%%%%%%%%%%%%%%%%%%%%%%%%%%%%%%%%%%%%%%%%%%%%%%%
	\beq
	\label{eq:LambdaR}
	i_{1*}:
	H_j(\ola^{2\pi p D}, 
	\ola^0; \mathbb{Q}) 
	\longrightarrow 
	H_j(\ola,
	\ola^0 ; \mathbb{Q}) 
	\eeq
%%%%%%%%%%%%%%%%%%%%%%%%%%%%%%%%%%%%%%%%%%%%%%%%%%%%%%%%%%%%%%
	induced by the inclusion is surjective for 
	$j \le 4p(m-1)+2$ since
	$4p(m-1)+2 <i(p+1)=(4p+2)(m-1).$
	From Assumption~\ref{assumption} and
Lemma~\ref{lem:morse} we conclude
%%%%%%%%%%%%%%%%%%%%%%%%%%%%%%%%%%%%%%%%%%%%%%%%%%%%%%%%%%%%%%
	\beq
	\label{eq:even}
	H_j(\ola^{ 2\pi p D}, 
	\ola^0; \mathbb{Q})=
	H_j(\ola^{ 2\pi p D^3}, 
	\ola^0; \mathbb{Q}) =0
	\eeq
%%%%%%%%%%%%%%%%%%%%%%%%%%%%%%%%%%%%%%%%%%%%%%%%%%%%%%%%%
	for all odd $j\le 4p(m-1)+2.$
	
	If the homomorphism given in
	Equation~\eqref{eq:LambdaR} is not injective for some
	$j\le 4p(m-1)+2$ then there is a non-trivial class
%%%%%%%%%%%%%%%%%%%%%%%%%%%%%%%%%%%%%%%%%%%%%%%%%%%%%%%%%
	\beqn
	Z \in H_j(\ola^{ 2\pi p D}S^n, 
	\ola^0 S^n; \mathbb{Q}) 
	\eeqn
	%%%%%%%%%%%%%%%%%%%%%%%%%%%%%%%%%%%%%%%%%%%%%%%%
	with $\deg(Z)=j\le 4p(m-1)+2$ such that
	$i_{1*}(Z)=0.$
	%%%%%%%%%%%%%%%%%%%%%%%%%%%%%%%%%%%%%%%%%%%
	%%%%%%%%%%%%%%%%%%%%%%%%%%%%%%%%%%%%%%%%%%%%
	%%%%%%%%%Ergaenzen!
	%%%%%%%%%%%%%%%%%%%%%%%%%%%%%%%%%%%%%%%%%%%
	%%%%%%%%%%%%%%%%%%%%%%%%%%%%%%%%%%%%%%%%%%%
	We consider the homomorphisms induced by
	the respective inclusions
	\begin{eqnarray*}
	\label{eq:inclusion1}	
	i_{2*}:
	H_j(\ola^{2\pi p D},\ola^0;\q) &\longrightarrow&
	H_j(\ola_0^{2\pi p D^2},\ola^0;\q)\\
	\label{eq:inclusion2}	
	i_{3*}:
	H_j(\ola^{2\pi p D^2}_0,\ola^0;\q) & \longrightarrow &
	H_j(\ola^{2\pi p D^3},\ola^0;\q)\\
	\label{eq:inclusion3}	
	i_{4*}:
	H_j(\ola^{2\pi p D^3},\ola^0;\q) &\longrightarrow &
	H_j(\ola,\ola^0;\q)\,.
\end{eqnarray*}
Then $i_{1*}=i_{4*}\circ i_{3*}\circ i_{2*}.$
Since the homomorphism
\begin{equation}
	\label{eq:inclusion4}	
	i_{4*} \circ i_{3*}:
	H_j(\ola^{2\pi D^2}_0,\ola^0;\q) \longrightarrow
	H_j(\ola,\ola^0;\q)
\end{equation}
is an isomorphism for all $j\le 4p(m-1)+2,$
cf. Equation~\eqref{eq:hom_inclusion}, we conclude
that $Z$ lies in the kernel of the homomorphism
\begin{equation}
	i_{3*}\circ i_{2*}:
H_j(\ola^{2\pi p D},\ola^0;\q) \longrightarrow
H_j(\ola^{2\pi p D^3},\ola^0;\q)\,,	
\end{equation}
	%%%%%%%%%%%%%%%%%%%%%%%%%%%%%%%%%%%%%%%%%%%%%%%%%
i.e. $(i_3\circ i_2)_*(Z)=0.$
The exactness of the 
	long homology sequence of
	the triple 
		$(\ola^{ 2\pi p D^3}S^n, 
	\ola^{2\pi p D}S^n, \ola^0 \sn)$
	implies that there exists
a non-trivial class 
	$$
	Y \in H_{j+1}(\ola^{ 2\pi p D^3}S^n, 
	\ola^{ 2\pi p D}S^n; \mathbb{Q}) 
	$$
	with $\partial_* Y=Z.$ 
	Here $\partial_*$ is the boundary operator of
	the long homology sequence of the triple.
	But since $j$ is even this leads to a contradiction
	to Equation~\eqref{eq:even}. 
\end{proof}

Let $\len=F(c)=l(c)$ be the length of the prime closed geodesic $c.$
Then we obtain for the Betti numbers
$b_j(c^r)= \rk H_j(\Lambda^{rL}, \Lambda^{(r-1)L} ;\z
)$
of the critical group of $c^r, r\le \lm/L:$
%%%%%%%%%%%%%%%%%%%%%%%%%%%%%%%%%%%%%%%%%%%%%%%%%%%%%%%%%%%%%%%%%%%%%%%%%
\beqn
%\label{eq:bkr}
b_k\left(c^r\right)=
\left\{
\begin{array}{ccc}
	1&;& k=\ind(c^r), r \mbox{ odd, or } \gamma_c=1\\
	1&;& k=\ind(c^r)+1, r \mbox{ odd, or } \gamma_c=1\\
	0&;& \mbox{ otherwise }
\end{array}
\right.\,.
\eeqn
%%%%%%%%%%%%%%%%%%%%%%%%%%%%%%%%%%%%%%%%%%
The Betti numbers
$\overline{b}_k(c^r)=\rk H_k(\ola^{rL}, 
\ola^{(r-1)L};\z)
=\dim H_k(\ola^{rL},
\ola^{(r-1)L};\q)$ of the 
$\sone$-critical group of $c^r$
for $r\le \lm/L:$
%%%%%%%%%%%%%%%%%%%%%%%%%%%%%%%%%%%%%%%%%%%%%%%%%%%%%%%%%%%%%%%%%%%%%
\beqn
%\label{eq:bkoverline}
\overline{b}_k(c^r)=
\left\{
\begin{array}{cclcc}
	1 & ; & k=\ind(c^r)&,& \gamma_c= 1 \mbox{ or } r \mbox{ odd }\\
	0 & ; &\mbox{ otherwise }&&
\end{array}
\right.
\,.
\eeqn
%%%%%%%%%%%%%%%%%%%%%%%%%%%%%%%%%%%%%%%%%%%%%%%%%%%%%%%%%%%%%%%%%%%%%%%
The Betti numbers 
$b_k=\rk H_k(\Lambda S^n, 
\Lambda^0 S^n;\z)=\dim
H_k(\Lambda S^n, 
\Lambda^0 S^n;\q)
$ are given by
%%%%%%%%%%%%%%%%%%%%%%%%%%%%%%%%%%%%%%%%%%%%%%%%%%%%%%%%%%%%%%%%%%%%%%%%
\beqn
b_k=
\left\{
\begin{array}{ccl}
	1 &;& k=2s(m-1), s\ge 1\\
	1 &;& k=2s(m-1)+1, s\ge 2\\
	0 &;& k \mbox{ otherwise }
\end{array}
\right.\,,
\eeqn
%%%%%%%%%%%%%%%%%%%%%%%%%%%%%%%%%%%%%%%%%%%%%%%%%%%%%%%%%%%%%%%%%%%%%%
cf. Equation~\eqref{eq:hlambdasm}.
The Betti numbers 
$\overline{b}_k=\rk H_k(\ola \sm, 
\ola^0 \sm;\z)$

$=
\dim
H_k(\ola \sm, 
\ola^0 \sm;\q)
$
of the 
$\sone-$quotient space are as follows:
%%%%%%%%%%%%%%%%%%%%%%%%%%%%%%%%%%%%%%%%%%%%%%%%%%%%%%%%%%%%%%%%%%%%%%%%
\beq
\label{eq:BettiSone}
\overline{b}_k=
\left\{
\begin{array}{ccl}
	2 &;& k=2s (m-1), s\ge 2\\
	1 &;& k \ge 2m-2, k \mbox{ even}, k \not= 2s(m-1), s\ge 2\\
	0 &;& k \mbox{ otherwise }
\end{array}
\right.\,,
\eeq
%%%%%%%%%%%%%%%%%%%%%%%%%%%%%%%%%%%%%%%%%%%%%%%%%%%%%%%%%%%%%%%%%%%%%%%%%
cf. Equation~\eqref{eq:homlasone}.

%%%%%%%%%%%%%
Bott's formula for the sequence $(\ind(c^r))_{r\ge 1}$
of indices of the iterates $c^r$ implies (cf. for example~\cite{Ra10}):
%%%%%%%%%%%%%%%%%%%%%%%%%%%%%%%%%%%%%%%%%%%%%%%%%%%%%%%%%%%%%%%%%%%%%
\beq
\label{eq:bott}
\ind(c^r)\ge \ind(c), r\ge 1\,.
\eeq
%%%%%%%%%%%%%%%%%%%%%%%%%%%%%%%%%%%%%%%%%%%%%%%%%%%%%%%%%%%%%%%%%%%%%%%%%%%%
Lemma~\ref{lem:morse}, Equation~\eqref{eq:bott} and
Equation~\eqref{eq:BettiSone} imply
that $\ind(c)=n-1$ and that the sequence $\ind(c^r)$
is monotone increasing, i.e. for all $r\ge 1:$
%%%%%%%%%%%%%%%%%%%%%%%%%%%%%%%%%%%%%%%%%%%%%%%%%%%%%%%%%%%%%%%%%%%%%%%%%%%
\beq
\label{eq:monotone}
\ind\left(c^{r+1}\right)\ge \ind\left(c^{r}\right)\,,
\eeq
cf. \cite{Ra17} or \cite{Ra10}.
%%%%%%%%%%%%%%%%%%%%%%%%%%%%%%%%%%%%%%%%%%%%%%%%%%%%%%%%%%%%%%%%%%%%%%%%%%
Bott's formula 
implies that
$v_j>0$ 
resp. $v_j^*>0$
for $j \le 4p(m-1)+2$ holds only for even $j.$
Since 
%%%%%%%%%%%%%%%%%%%%%%%%%%%%%%%%%%%%%%%%%%%%%%%%%%%%%%%%
\beq
\label{eq:morse-odd}
v_j=v_j^*=\overline{\beta}_j=\overline{\beta}_j^*=0
\eeq
%%%%%%%%%%%%%%%%%%%%%%%%%%%%%%%%%%%%%%%%%%%%%%%%%%%%%%%%%%%%%%
for all odd $j\le 4p(m-1)+2$ the Morse inequalities
take
the simple form for all $j \le 4p(m-1)+2,$
cf. Equation~\eqref{eq:morse2}:
%%%%%%%%%%%%%%%%%%%%%%%%%%%%%%%%%%%%%%%%%%%%%%%%%%%%%%%%%%%%%%
\beq
\label{eq:morse}
v_j=\overline{\beta}_j=\overline{b}_j\,;\,v_j^*=\overline{\beta}_j^*\,.
\eeq
%%%%%%%%%%%%%%%%%%%%%%%%%%%%%%%%%%%%%%%%%%%%%%%%%%%%%%%%%%%%%%%%%%%%%%
If $\gamma_c=1/2,$ i.e.
if $\ind c^2=2m-1=\ind c +1,$ we obtain from
Bott's formula for $\ind(c^r)$ that the
sequence
$\ind(c^r), r\ge 1$ is strictly monotone increasing.
But  $\overline{b}_{4m-4}=2,$
cf. Equation~\eqref{eq:BettiSone}.
This contradicts 
Equation~\eqref{eq:morse}.
Hence $\gamma_c=1,$ resp. $\ind(c^2)=2m.$
The 
sequence $\ind(c^r)_{1\le r\le L_1/L}$ is uniquely determined by
Equation~\eqref{eq:monotone} and Equation~\eqref{eq:morse}: 
%%%%%%%%%%%%%%%%%%%%%%%%%%%%%%%%%%%%%%%%%%%%%%%%%%%%%%%%%%%%%%%%%%%%%%
\begin{eqnarray*}
	\left(\ind(c^r)\right)_{r\ge 1}=&
		\\ 
&	\left(2m-2, 2m, 2m+2,\ldots, 4m-6, 4m-4, 4m-4, 4m-2,\ldots\right.\\
%	\nonumber
&	\left.
	\ldots, 6m-8, 6m-6, 6m-6, 6m-4,\ldots\right)
\end{eqnarray*}
%%%%%%%%%%%%%%%%%%%%%%%%%%%%%%%%%%%%%%%%%%%%%%%%%%%%%%%%%%%%%%%%%%%%%%%
%%%%%%%%%%%%%%%%%%%%%%%%%%%%%%%%%%%%%%%%%%%%%%%%%%%%%%%%%%%%%%%%%%
%%%%%%%%%%%%%%%%%%%%%%%%%%%%%%%%%%%%%%%%%%%%%%%%%%%%%%%%%%%%%%%%%%%%%%%%%
From Lemma~\ref{lem:two} and Equation~\eqref{eq:BettiSone} we conclude
$$
\overline{\beta}_{4p(m-1)}=
\dim H_{4p(m-1)}(
\ola^{2\pi p D},\ola^0;\q
)=
\overline{b}_{4p(m-1)}=2\,.
$$
Therefore 
we obtain the following,
cf. \cite[Eq.(13)]{Ra17}:
%%%%%%%%%%%%%%%%%%%%%%%%%%%%%%%%%%%%%%%%%%%%%%%%%%%%%%%%%%%%%%
\begin{lemma}
	\label{lem:3.6}
	If Assumption~\ref{assumption} holds then for 
	$p=\prm$ we have
%%%%%%%%%%%%%%%%%%%%%%%%%%%%%%%%%%%%%%%%%%%%%%%%%%%%%%%%%%%%%%
	\beq
	\label{eq:ckm}
	\ind\left(c^{(2p-1)m}\right)=
	\ind\left(c^{(2p-1)m+1}\right)=4p(m-1)\,.
	\eeq
	%%%%%%%%%%%%%%%%%%%%%%%%%%%%%%%%%%%%%%%%%%%%%%%%%%%%%%%%%%%%
	and $L(c^{(2p-1)m+1})=((2p-1)m+1)L\le 2\pi p D.$
\end{lemma}
%%%%%%%%%%%%%%%%%%%%%%%%%%%%%%%%%%%%%%%%%%%%%%%%%%%%%%%%%%%%%%%%%%%%%
%%%%%%%%%%%%%%%%%%%%%%%%%%%%%%%%%%%%%%%%%%%%%%%%%%%%%%%%%%%%%%%%%%%%%
%%%%%%%%%%%%%%%%%%%%%%%%%%%%%%%%%%%%%%%%%%%%%%%%%%%%%%%%%%%%%%%%%%%%%
\begin{lemma}	
	\label{lem:hlambda_iso}
If Assumption~\ref{assumption} holds, $p=\prm$ then 
for all $j \le 4p (m-1)+1:$
%%%%%%%%%%%%%%%%%%%%%%%%%%%%%%%%%%%%%%%%%%%%%%%%%%%%%%%%%%%%%%%%%%%%%%%%%%%%%%%%%%%%%%
\beq
\label{eq:holaD}
H_j ( 
\la^{2\pi p D^3}, \la^{2\pi p D};\q
)
=0\,;\,
H_j ( 
\ola^{2\pi p D^3}, \ola^{2\pi p D};\q
)
=0\,,
\eeq
%%%%%%%%%%%%%%%%%%%%%%%%%%%%%%%%%%%%%%%%%%%%%%%%%
and the homomorphism 
%%%%%%%%%%%%%%%%%%%%%%%%%%%%%%%%%%%%%%%%%%%%%%%%%%
\beq
\label{eq:lambda-iso}
H_j ( 
\la^{2\pi p D}, \la^{0};\q
)
\longrightarrow 
H_j ( 
\la , \la^0,\q
)
\eeq
%%%%%%%%%%%%%%%%%%%%%%%%%%%%%%%%%%%%%%%%%%%%%%
induced by the inclusion is an isomorphism
for all $j\le 4p(m-1).$
\end{lemma}
%%%%%%%%%%%%%%%%%%%%%%%%%%%%%%%%%%%%%%%%%%%%%%%%%%%%%%%%%%%%%%%%%%%%%
\begin{proof}
Since $\ind(c^r)\ge 4p(m-1)+2$
for
all $r\ge (2p-1)m+2$ it also follows that for
$j \le 4p(m-1)+1:$
$v_j=v_j^*,$ hence Equation~\eqref{eq:holaD}
follows, cf. Lemma~\ref{lem:3.6}.
The inclusion Equation~\eqref{eq:D} 
together with the isomorphism~\eqref{eq:hlambdaglobal}
imply that the homomorphism~\eqref{eq:lambda-iso}
is an isomorphism for $j \le 4p(m-1).$
\end{proof}
%%%%%%%%%%%%%%%%%%%%%%%%%%%%%%%%%%%%%%%%%%%%%%%%%%%%%%%%%%%%%%%%%%%
%%%%%%%%%%%%%%%%%%%%%%%%%%%%%%%%%%%%%%%%%%%%%%%%%%%%%%%%%%%%%%%%%
%%%%%%%%%%%%%%%%%%%%%%%%%%%%%%%%%%%%%%%%%%%%%%%%%%%%%%%%%%%%%%%%%
%%%%%%%%%%%%%%%%%%%%%%%%%%%%%%%%%%%%%%%%%%%%%%%%%%%%%%%%%%%%%%%%%
\section{Proof of Theorem~\ref{thm:length}}
%%%%%%%%%%%%%%%%%%%%%%%%%%%%%%%%%%%%%%%%%%%%%%%%%%%%%%%%%%%%%%%%
\label{sec:proofs}
%%%%%%%%%%%%%%%%%%%%%%%%%%%%%%%%%%%%%%%%%%%%%%%%%%%%%%%%%%%%%%%%%%%
	%%%%%%%%%%%%%%%%%%%%%%%%%%%%%%%%%
	%%%%%%%%%%%%%%%%%%%%%%%%%%%%%%%%
	In this proof we use as coefficient ring for homology the
	ring $\mathbb{Z}$ of integers if not otherwise stated.
	We assume that Assumption~\ref{assumption} holds and derive a contradiction.
	Let $p=\prm.$ Because of the Morse inequalities~\eqref{eq:morse} 
	and Lemma~\ref{lem:two}
	we obtain
	for $j \le 4 p(m-1)+2:$
	%%%%%%%%%%%%%%%%%%%%%%%%%%%%%%%%%%%%%%%%%%%%%%%%%%%%%%%%%%%%%%%%%%%%%%%%%%
	\beqn
	FH_j(\ola,\ola^0;\z)
	\cong
	FH_j(\ola^{2\pi D p},\ola^0
	)
	\cong 
	\bigoplus_{r L \le 2\pi \prm D} FH_j(\ola^{rL}, 
	\ola^{(r-1)L})\,.
	\eeqn
	%%%%%%%%%%%%%%%%%%%%%%%%%%%%%%%%%%%%%%%%%%%%%%%%%%%%%%%%%%%%%%%%%%%%%%%%%%%%
	We have shown that $\ind(c)=2m-2,\ind(c^2)=2m$
	and $\ind(c^{(2p-1)m})=\ind(c^{(2p-1)m+1})=4p(m-1),$
	cf. Equation~\eqref{eq:ckm}.
	It also follows that $((2p-1)m+1)L\le 2\pi p D.$
	Let
	$s_{(2p-1)m}, s_{(2p-1)m+1}$  denote generators
	of the local critical groups, cf. Equation~\eqref{eq:sr}.
	It follows that
	%%%%%%%%%%%%%%%%%%%%%%%%%%%%%%%%%%%%%%%%%%%%
	\begin{eqnarray*}
	H_{4p(m-1)}\left(
	\la^{((2p-1)m+1)L},\la^{((2p-1)m-1)L}	
	\right)&\\
	\cong
	H_{4p(m-1)}\left(
	\la^{((2p-1)m+1)L},\la^{(2p-1)mL}\right)
	&\oplus
	H_{4p(m-1)}\left(
	\la^{(2p-1)mL},\la^{((2p-1)m-1)L} \right)\\
	&\cong 
	\z\cdot s_{(2p-1)m+1}
	\oplus
	\z\cdot s_{(2p-1)m}\,.
	\end{eqnarray*}
%%%%%%%%%%%%%%%%%%%%%%%%%%%%%%%%%%%%%%%%%%%%%%%%%%%%%%%%	
	We consider the following commutative diagram,
	the vertical homomorphisms are induced by inclusions, the horizontal ones by
	the canonical projection with respect to
	the $S^1$-action:
	%%%%%%%%%%%%%%%%%%%%%%%%%%%%%%%%%%%%%%%%%%%%%%%
\begin{equation*}
\begin{xy}
\xymatrix{
	H_{4p(m-1)}(
	\Lambda^{((2p-1)m+1)L},
	\Lambda^{((2p-1)m-1)L}
	)\ar[r]^{\rho_{1*}}    &   
	FH_{4p(m-1)}\left(
	\ola^{((2p-1)m+1)L},
	\ola^{((2p-1)m-1)L}
	\right)	 \\
	H_{4p(m-1)}(
	\Lambda^{((2p-1)m+1)L},
	\Lambda^{0}
	)\ar[r]^{\rho_{2*}} 
	\ar[u]_{h_{1*}}\ar[d]_{\cong}^{h_{2*}}
	&   
	FH_{4p(m-1)}\left(
	\ola^{((2p-1)m+1)L},
	\ola^{0}
	\right)\ar[d]_{\cong}^{j_{2*}}
	\ar[u]_{\cong}^{j_{1*}}\\
	H_{4p(m-1)}(
	\Lambda,
	\Lambda^{0}
	)\ar[r]^{\rho_*} 
	&   
	FH_{4p(m-1)}(
	\ola,
	\ola^{0}
	)
}	 
\end{xy}
\end{equation*}
%%%%%%%%%%%%%%%%%%%%%%%%%%%%%%%%%%%%%%%%%%%%%%%%%%%%%%%%%%%%%%%%%%%
$j_{1*},j_{2*}$ are isomorphisms, which follows from the following 
arguments:
%%%%%%%%%%%%%%%%%%%%%%%%%%%%%%%%%%%%%%%%%%%%%
For $r\le (2p-1)m-1:
\ind(c^r)\le 4p(m-1)-2,$
hence
$H_j(\ola^{((2p-1)m-1)L},
\ola^0)=0$ for $j=4p(m-1), 4p(m-1)-1.$
Therefore $j_{1*}$ is an isomorphism.
The homomorphism
$j_{2*}$ is an isomorphism since
$FH_{j}(\ola,
\ola^{((2p-1)m+1)L})=0$
for $j=4p(m-1), 4p(m-1)+1,$
cf.
Equation~\eqref{eq:morse}
and Equation~\eqref{eq:morse-odd}.

%%%%%%%%%%%%%%%%%%%%%%%%%%%%%%%%%%%%%%%%%%%%%%%%%%%%%%%%%%
Lemma~\ref{lem:hlambda_iso} implies that
$h_{2*}$ is an isomorphism.

Since $H_{4p(m-1)}(\la^{((2p-1)m-1)L},\la^0)=0$
the map $h_{1*}$ is injective and the image
of a generator $a_{p}'\in
H_{4p(m-1)}(\la^{((2p-1)m+1)L},\la^0)
\cong \z$  is a prime element.
Let $a_p \in H_{4p(m-1)}\left(
\la,\la^0\right)$ be a generator.
Then we conclude from 
Lemma~\ref{lem:hom-projection}:
%%%%%%%%%%%%%%%%%%%%%%%%%%%%%%%%%%%%%%%%%%%%%
\begin{equation*}
	\rho_*(a_p)=p \tilde{a}_p	
\end{equation*}
%%%%%%%%%%%%%%%%%%%%%%%%%%%%%%%%%%%%%%%%%%%%
for a generator $\tilde{a}_p\in 
H_{4p(m-1)}(\ola,\ola^0).$

Define $a_p'=(h_{2*})^{-1}a_p$ and for
$\eta \in \{0,1\}:
s_{(2p-1)m+\eta}'=(j_{2*})^{-1}s_{(2p-1)m+\eta},
s_{(2p-1)m+\eta}''=(j_{1*})^{-1}s_{(2p-1)m+\eta}'.$ Then there
are coprime integers
$\alpha,\beta$
% % % % % % % % % % % % % % % % % % % % % % % %
\begin{equation*}
	h_{1*}(a_p')=\alpha s_{(2p-1)m+1}+
	\beta s_{(2p-1)m}.
\end{equation*}
% % % % % % % % % % % % % % % % % % % % % % % %
Here we also allow $\alpha=0,$ which
implies $\beta=\pm 1$ resp.
$\beta=0,$ which implies $\alpha=\pm 1.$
By Equation~\eqref{eq:psr} we have
%%%%%%%%%%%%%%%%%%%%%%%%%%%%%%%%%%%%%%%%%%%%%%%%%%%%%%%
\begin{equation*}
\rho_{1*}(s_{(2p-1)m+1})=
((2p-1)m+1) \overline{s}_{(2p-1)m+1}\,;\,
\rho_{1*}\left(s_{(2p-1)m}\right)=
(2p-1)m\overline{s}_{(2p-1)m}.	
\end{equation*}
%%%%%%%%%%%%%%%%%%%%%%%%%%%%%%%%%%%%%%%%%%%%%%%%%%%%%%%%%%%%%%%%%%%
Since
$\overline{s}_{(2p-1)m+1}'',
 \overline{s}_{(2p-1)m}''$ form a basis for
 $FH_{4p(m-1)}(\ola,\ola^0)$ there are integers
 $w,z \in \z$ with
 $$
\tilde{a}_p=
w \overline{s}_{(2p-1)m+1}''+z\overline{s}_{(2p-1)m}''\,.
$$
%%%%%%%%%%%%%%%%%%%%%%%%%%%%%%%%%%%%%%%%%%%%%%%%%%%%%%%%%%%%%%%%%%%
We obtain the following explicit description of
the last commutative diagram with respect to the
given basis elements
%%%%%%%%%%%%%%%%%%%%%%%%%%%%%%%%%%%%%%%%%%%%%%%
	\begin{equation*}
	\begin{xy}
	\xymatrix @C=8pc @R=6pc{
		\z s_{(2p-1)m+1}\oplus \z s_{(2p-1)m}
		\ar[r]_{\rho_{1*}}^{
			\tiny
			\left(
			\begin{array}{cc} (2p-1)m+1&0\\0&
			(2p-1)m
			\end{array}
			\right)}   &   
		\z \overline{s}_{(2p-1)m+1}\oplus 
		\z \overline{s}_{(2p-1)m}	 \\
		\z a'_p\ar[r]^{\rho_{2*}} 
		\ar[u]_{h_{1*}}^{\small
			\left(
			\begin{array}{cc} 
			\alpha, &\beta
			\end{array}
			\right)}
		\ar[d]_{\cong}^{h_{2*}}
		&   
		\z  \overline{s}'_{(2p-1)m+1}\oplus 
		\z \overline{s}'_{(2p-1)m}
		\ar[d]_{j_{2*}}^{
			\tiny
			\left(
			\begin{array}{cc} 1&0\\0&1
			\end{array}
			\right)}
		\ar[u]_{j_{1*}}^{
			\tiny
			\left(
			\begin{array}{cc} 1&0\\0&1
			\end{array}
			\right)}\\
		\z a_p
		\ar[r]^{\rho_*} 
		&   
		\z \overline{s}''_{(2p-1)m+1}\oplus 
		\z \overline{s}''_{(2p-1)m}
	}	 
	\end{xy}
	\end{equation*}	
%%%%%%%%%%%%%%%%%%%%%%%%%%%%%%%%%%%%%%%%%%%%%%%%%%%%%%%%%%%%%%%%%%%%
	We conclude from this diagram
	\begin{eqnarray*}
		\rho_*(a_p)&=& p \tilde{a}_p=
		p\left(w \overline{s}_{(2p-1)m+1}''+
		z \overline{s}_{(2p-1)m}''\right)\\
		&=&j_{2*}\rho_{2*}h_{2*}^{-1} (a_p)=
		j_{2*}\rho_{2*}(a_p')=
		j_{2*}j_{1*}^{-1}\rho_{1*}h_{1*}(a_p')\\
		&=&
		j_{2*}j_{1*}^{-1}\rho_{1*}
		\left(\alpha\cdot s_{(2p-1)m+1}+
		\beta \cdot s_{(2p-1)m}\right)\\
		&=&
		\alpha ((2p-1)m+1)\cdot \overline{s}_{(2p-1)m+1}''+
		\beta (2p-1)m\cdot 		\overline{s}_{(2p-1)m}''\,.	
		\end{eqnarray*}
	%%%%%%%%%%%%%%%%%%%%%%%%%%%%%%%%%%%%%%%%%%%%%%%%
	Since
	$\overline{s}_{(2p-1)m+1}'', \overline{s}_{(2p-1)m}''$
	form a basis we obtain:
	%%%%%%%%%%%%%%%%%%%%%%%%%%%%%%%%%%%%%%%%%%%%%%%
	\beqn
	p w= ((2p-1)m+1) \alpha \enspace;\enspace 
	p z= (2p-1)m \beta
	\eeqn
	%%%%%%%%%%%%%%%%%%%%%%%%%%%%%%%%%%%%%%%%%%%%%%%%
	which is equivalent to
	%%%%%%%%%%%%%%%%%%%%%%%%%%%%%%%%%%%%%%%%%%%%%%%%%%
	\beq
	\label{eq:rplus1}
	p (2m\alpha-w)=(m-1)\alpha
	\enspace;\enspace
	p (2m\beta-z)=m \beta.
	\eeq
%%%%%%%%%%%%%%%%%%%%%%%%%%%%%%%%%%%%%%%%%%%%%%%%%%%%%%%%%%%
	Equation~\eqref{eq:rplus1} implies that
	$p$ is a common divisor of the numbers $\alpha$ and $\beta$ since $p=\prm$ neither 
	divides $m$ nor $m-1,$ cf. Lemma~\ref{lem:pm}.
	
	But the numbers 	$\alpha,\beta$
	are by assumption coprime, hence we arrive at
	a contradiction.
	Note that this argument is also	valid for the cases
	$\alpha=0, \beta=\pm 1,$ resp. 	$\alpha=\pm 1, \beta=0.$
%%%%%%%%%%%%%%%%%%%%%%%%%%%%%%%%%%%%%%%%%%%%%%%%%%%%%%%%%%%%%%%%
%%%%%%%%%%%%%%%%%%%%%%%%%%%%%%%%%%%%%%%%%%%%%%%%%%%%%%%%%%%%%%%%
\section{Katok metrics}
%%%%%%%%%%%%%%%%%%%%%%%%%%%%%%%%%%%%%%%%%%%%%%%%%%%%%%%%%%%%%%%
\label{sec:katok}
%%%%%%%%%%%%%%%%%%%%%%%%%%%%%%%%%%%%%%%%%%%%%%%%%%%%%%%%%%%%%%%%
	Choose
	numbers $p_1<\ldots<p_m$ which are
	relatively prime and let $p=p_1\cdots p_m.$ 

	Let 
	\begin{equation*}
	R(\phi)	=
	\left(
	\begin{array}{cc}
	\cos  \phi &-\sin \phi\\
	\sin \phi &\cos \phi
\end{array}
	\right)	
	\end{equation*}
be the rotation in $\R^2$ with angle
$\phi.$ Let $\R^{2m}= V_1\oplus \ldots \oplus V_m$
be an orthogonal decomposition into
$2$-dimensional subspaces and let
$A(t)\in S\mathbb{O}(2m)$ be the 
$1$-parameter family of isometries
of $S^{2m-1}$ with
$A(t)|V_j=R(pt/p_j), j=1,\ldots,m.$
This one-parameter group of isometries
generates a Killing field $V$ on $S^{2m-1}$
with norm
$$
\|V\|
=a(p_1,\ldots,p_m)=
p\left(\sum_{j=1}^m p_j^{-2}\right)^{1/2}\,.
$$
For $\mu < a(p_1,\ldots,p_m)^{-1}$ we define the
Killing field $V_{\mu}=\mu V$
with $\|V_{\mu}\|<1.$
Then the sphere bundle
determined by
$\{\xi \in T_xS^n\,;\,
\|\xi -V_{\mu}(x)\|=1\}$
determines the unit sphere bundle of 
a non-reversible Finsler metric
$f_{\mu}.$ These metrics are called 
\emph{Katok metrics,}
cf. \cite[p.139]{Zi}
or \emph{Zermelo deformation} of
the standard metric~\cite{FM}. For the flow 
$\psi_t$ of the Killing field $V_{\mu}$
the geodesics of the Finsler metric $f_{\mu}$
are of the form $t \longmapsto \psi_t(c(t))$
for a great circle $c(t)$ on $S^n.$
For irrational $\mu$ the Katok metric
$f_{\mu}$ determined by the Killing field
$V_{\mu}$ has exactly $2m$ closed geodesics 
$c_j^{\pm},j=1,\ldots,m$ with
$c_j^{-}(t)=c_j^{+}(-t)$ for all 
$t,$ 
cf.~\cite[p.139]{Zi}
and \cite{FM}. 
These
 are the great circles invariant
under the flow $\psi_t.$ 
Since $\mu$ is irrational
these metrics are bumpy. 
As remarked in \cite{Ra04}, \cite{FM} these metrics
have constant flag curvature
$1.$ 
The lengths $l(c_j^{\pm}), j=1,\ldots,m$
of the closed geodesics $c_j^{\pm}$
are given
by
%%%%%%%%%%%%%%%%%%%%%%%%%%%%%%%%%%%%%%%%%%
\beqn
l(c_j^{\pm})=
\frac{2\pi}{1\pm \frac{p}{p_j}\mu}\,,
j=1,\ldots,m\,.
\eeqn
%%%%%%%%%%%%%%%%%%%%%%%%%%%%%%%%%%%%%%%%%%%%%%%%%%%%%
The distortion is given by:
%%%%%%%%%%%%%%%%%%%%%%%%%%%%%%%%%%%%%%%%%%%%%%%%%%
\beqn
D=\lambda=\frac{1}{1-\max\{\|V_{\mu}(x)\|, x\in \sn\}}=
\frac{1}{1-\mu a(p_1,\ldots,p_m)}\,.
\eeqn
%%%%%%%%%%%%%%%%%%%%%%%%%%%%%%%%%%%%%%%%%%%%%%%%%
If we choose $p_1=1$ then
for an arbitrary $N\in \n$ one can choose
$N<p_2<\ldots<p_m $ and an irrational $\mu$ satisfying
%%%%%%%%%%%%%%%%%%%%%%%%%%%%%%%%%%%%%%%%%
\beqn
\frac{1}{p}\frac{N-1}{N} <\mu 
<\frac{1}{a(1,p_2,\ldots,p_m)}\,.
\eeqn
%%%%%%%%%%%%%%%%%%%%%%%%%%%%%%%%%%%%%%%%%%
This implies that $l(c_1^-)\ge 2\pi N,
\ind(c_1^{-1})\ge 2N(m-1)$ and the distortion
satisfies $D\ge N.$ 
One can also show that
$l(c_j^+)<2\pi, \ind(c_j^+)\le 4(m-1), 1\le j\le m$
and for $2\le j\le m$ we obtain
$l(c_j^{-1})<2\pi N/(N-1),
\ind(c_j^{-1})\le 6(m-1).$

For $n=3,m=2$ we obtain 
for any $N$ a bumpy Katok metric $f_{\mu}$
with exactly
four closed geodesics $c_1^{\pm},c_2^{\pm}$
with the following (in)equalities for the indices
resp. lengths of  these closed
geodesics:
$l(c_1),l(c_2)<2\pi,
\ind(c_1)=2,\ind(c_2)=4,
l(c_2^{-})\le 2\pi N/(N-1); \ind(c_2^{-1}) \in \{4,6\},
$
and $l(c_1^{-1})\ge 2\pi N, \ind(c_1^{-1})\ge
2N(m-1).$

\bigskip

In a certain sense one can say that
these examples show that
the minimal number of 
\emph{short} closed geodesics on a 
sphere of dimension $n=2m-1$ resp.
$n=2m$ is $2m-1.$ Here \emph{short} closed geodesics
posess an \emph{a priori} bound for the index.
%%%%%%%%%%%%%%%%%%%%%%%%%%%%%%%%%%%%%%%%%%%%%%%%%%%%%%%%%%%%%%%%%%%%%%%%%%%
%%%%%%%%%%%%%%%%%%%%%%%%%%%%%%%%%%%%%%%%%%%%%%%%%%%%%%%%%%%%%%%%%%%%%%%%%%%
%%%%%%%%%%%%%%%%%%%%%%%%%%%%%%%%%%%%%%%%%%%%%%%%%%%%%%%%%%%%%%%%%%%%%%%%%%%
\section{Genericity statement}
%%%%%%%%%%%%%%%%%%%%%%%%%%%%%%%%%%%%%%%%%%%%%%%%%%%%%%%%%%%%%%%%%%%%%%%%%%%%%
\label{sec:generic}
%%%%%%%%%%%%%%%%%%%%%%%%%%%%%%%%%%%%%%%%%%%%%%%%%%%%%%%%%%%%%%%%%%%%%%%%%%%%%%
	%\label{rem:open-and-dense}	
	The set $\mathcal{F}(T)$
	of Finsler metrics on a compact manifold
	$M$ for which all closed geodesics of 
	length $\le T$ are non-degenerate is an open and
	dense subset
	of the space $\mathcal{F}=
	\mathcal{F}(M)$ of Finsler metrics
	on $M$
	with the strong $C^r$-topology for
	$r\ge 4,$ 
	cf.\cite[Thm.4]{RT2020}.
	%%%%%%%%%%%%%%%%%%%%%%%%%%%%%%%%%%%%%%%%%%%%%%%%%%%%%%%%%%%%%%%%%%%%%%%
	%%%%%%%%%%%%%%%%%%%%%%%%%%%%%%%%%%%%%%%%%%%%%%%%%%%%%%%%%%%%%%%%%%%%%%%
	%%%%%%%%%%%%%%%%%%%%%%%%%%%%%%%%%%%%%%%%%%%%%%%%%%%%%%%%%%%%%%%%%%%%%%%
	\begin{proof}[Proof of Theorem~\ref{thm:open-and-dense}]
	Let $f_1 \in \mathcal{F}_1(L),$ hence 
	by definition
	all closed geodesics of the Finsler metric
	$f_1$ of length $\le D^3(f_1) L$ are non-degenerate.
	Let $\phi_{f_1}^t: TM \longrightarrow TM$ be the \emph{geodesic flow}
	of the Finsler metric $f_1.$ 
	If $\tau: TM\longrightarrow M$ is the tangent bundle
	projection then $t \longmapsto \tau(\phi_{f_1}^t(v))$
	is the geodesic $c_v$ determined by the initial condition
	$c_v'(0)=v.$
	Let $HM=(TM-M)/\R^+$ be the bundle of oriented directions in the tangent bundle $TM.$ We consider instead of the geodesic flow
	$\phi_{f_1}^t: TM \longrightarrow TM$ the map
	$\Phi_{f_1}^t: HM \longrightarrow HM$
	with $\Phi_{f_1}^t(v)=\phi_{f_1}^t(v/f_1(v)),$ hence the
	geodesic $t \in \R \longmapsto \tau(\Phi_{f_1}^t(v))\in M$
	is parametrized by arc length. If 
	$\Phi_{f_1}^t(v)$ is a periodic flow line of period
	$a,$ then $t \in [0,a]\longmapsto \tau(\Phi_{f_1}^t(v))$
	is a closed geodesic of length $a.$
	The \emph{minimal period} is then the length of the
	underlying prime closed geodesic.	 

	Let $\Phi_{f_1}^t(v_1),\ldots, \Phi_{f_1}^t(v_N)$ 
	be the periodic flow lines of the
	geodesic flow 
	$\Phi_{f_1}^t: HM \longrightarrow HM$
	corresponding to the closed geodesics of period (resp. length)
	$a_1,\ldots,a_N$ which satisfy $a_i\le D(f_1)^3 L.$ 
		
	Then there is an
	open neighborhood $\mathcal{U}\subset \mathcal{F}$ of $f_1$ such
	that the following holds: There are continuous maps
	$v_i: f \in \mathcal{U} \mapsto v_i(f) \in H\sn\,, \,a_i:f \in \mathcal{U}
	\mapsto a_i(f) \in (0,\infty), i=1,2,\ldots,N$ 
	with $v_i=v_i(f_1), a_i=a_i(f_1), i=1,\ldots,N$
	such that for all $f \in \mathcal U$ the  sets
	$\Phi_{f}^t(v_i(f)), t\in [0,a_i(f)], i=1,2,\ldots,N$
	are periodic 
	and non-degenerate
	flow lines of the geodesic flow of $f$ 
	of period $a_1(f),\ldots,a_N(f)$ and there are no further periodic flow
	lines of $f$ of length $\le D^3(f) L.$ This holds since the distortion
	% % % % % % % % % % % % % % % % % % % % % %
	\beq
	\label{eq:Df}
	f \in \mathcal{F}\longmapsto D(f) \in (0,\infty)
	\eeq
	% % % % % % % % % % % % % % % % % % % % % % % % %
	is a continuous function.
	Hence the set $\mathcal{F}_1(L)$ is an open subset of $\mathcal{F}.$
	
	\smallskip
	
	%Choose $f_1\in \mathcal{F}_1(L)$ and 
	Choose $T=2 D^3(f) L.$ 
	Since $\mathcal{F}(T)$ is a dense subset of $\mathcal{F}$ we find a
	sequence $(f_k)_{k\ge 2} \subset \mathcal{F}(T)$ converging to $f_1.$
	Since the function given in Equation~\eqref{eq:Df} is continuous
	it follows that also $\mathcal{F}_1(L)$ is dense in $\mathcal{F}.$	
	\end{proof}
 %%%%%%%%%%%%%%%%%%%%%%%%%%%%%%%%%%%%%%%%%%%%%%%%%%%%%%%%%%%%%%%%%%%%%%%%%%%%%%%%%%%%%%%%%%%%%%%%%%%%%%%%%%%%%%%%%%%%%%%%%%%%%%%%%%%%%%%%%%
 \section*{Statements and Declarations}
 %{\em Conflicts of interest/competing interests}\\
 The author has no relevant financial or non-financial interests
 to disclose.
 
 \smallskip
 
 %{\em Datasharing}\\
 Data sharing is not applicable to this article as no datasets were generated or analysed during the current study.

 %%%%%%%%%%%%%%%%%%%%%%%%%%%%%%%%%%%%%%%%%%%%%%%%%%%%%%%%%%%%%%%%%%%%%%%%%

\end{document}